\documentclass[11pt]{article}
\usepackage{amsmath, amssymb, amsthm}

\date{}
\title{\vspace{-0.9cm}Most probably intersecting hypergraphs}
\author{
Shagnik Das \thanks{Department of Mathematics, UCLA, Los Angeles, CA, 90095. Email: shagnik@ucla.edu.}
\and
Benny Sudakov \thanks{Department of Mathematics, ETH, 8092 Zurich, Switzerland, and Department of Mathematics, UCLA, Los Angeles, CA 90095, USA. Email: benjamin.sudakov@math.ethz.ch. Research supported in part by SNSF grant 200021-149111 and by a USA-Israel BSF grant.}
}

\allowdisplaybreaks

\theoremstyle{plain}
\newtheorem{THM}{Theorem}[section]
\newtheorem{PROP}[THM]{Proposition}
\newtheorem{LEMMA}[THM]{Lemma}

\newtheorem{COR}[THM]{Corollary}

\theoremstyle{definition}

\newcommand{\prob}{\mathbb{P}}

\newcommand{\card}[1]{\left| #1 \right|}

\newcommand{\cC}{\mathcal{C}}
\newcommand{\cF}{\mathcal{F}}
\newcommand{\cG}{\mathcal{G}}
\newcommand{\cH}{\mathcal{H}}

\newcommand{\cL}{\mathcal{L}}
\newcommand{\cM}{\mathcal{M}}

\newcommand{\inter}{\mathrm{int}}

\begin{document}
\maketitle

\begin{abstract}
The celebrated Erd\H{o}s-Ko-Rado theorem shows that for $n \ge 2k$ the largest intersecting $k$-uniform set family on $[n]$ has size $\binom{n-1}{k-1}$.  It is natural to ask how far from intersecting larger set families must be.  Katona, Katona and Katona introduced the notion of most probably intersecting families, which maximise the probability of random subfamilies being intersecting.

We study the most probably intersecting problem for $k$-uniform set families.  We provide a rough structural characterisation of the most probably intersecting families and, for families of particular sizes, show that the initial segment of the lexicographic order is optimal.
\end{abstract}

\section{Introduction} \label{sec:intro}

A family of sets $\cF$ is said to be \emph{intersecting} if $F_1 \cap F_2 \neq \emptyset$ for all $F_1, F_2 \in \cF$.  A central result in extremal set theory is the Erd\H{o}s-Ko-Rado theorem, which determines the largest size of an intersecting $k$-uniform family over $[n]$.  Given this extremal result, one may then investigate the appearance of disjoint pairs in larger families of sets.

Recently Katona, Katona and Katona introduced a probabilistic version of this supersaturation problem.  Given a set family $\cF$, let $\cF_p$ denote the random subfamily obtained by keeping each set independently with probability $p$.  They asked, for a given $p$, $n$ and $m$, which set families on $[n]$ with $m$ sets maximise the probability of $\cF_p$ forming an intersecting family.  We study this problem for $k$-uniform set families.  In the case $k=2$, we determine the optimal graphs when they are not too dense.  In the hypergraph setting, we provide an approximate structural result, and are able to determine the extremal hypergraphs exactly for some ranges of values of $m$.  These mark the first general results for the probabilistic supersaturation problem for $k$-uniform set families.

We now discuss the history of the supersaturation problem for intersecting families, before introducing the probabilistic version of Katona, Katona and Katona and presenting our new results.

\subsection{Supersaturation for intersecting families} \label{subsec:supersat}

The typical extremal problem asks how large a structure can be without containing some forbidden configuration.  A classic example is that of the Erd\H{o}s-Ko-Rado Theorem \cite{ekr61}, which shows that when $n \ge 2k$ the largest $k$-uniform intersecting family of sets on $[n]$ has size $\binom{n-1}{k-1}$.  This theorem has been central to extremal set theory, inspiring many extensions and applications over the years; see Anderson's book \cite{and87} for a brief survey.

One may rephrase the extremal question as asserting that larger structures must contain a forbidden configuration. This naturally leads to what is known as a supersaturation problem: how many such configurations appear in larger structures?  In the context of intersecting families, this amounts to asking how many disjoint pairs of sets must appear in large families.  In the non-uniform setting, this was first studied by Frankl \cite{frankl77} and, independently, Ahlswede \cite{ahlswede80}, who showed that the number of disjoint pairs is minimised by taking sets as large as possible.

\begin{THM}[Frankl \cite{frankl77}, 1977; Ahlswede \cite{ahlswede80}, 1980] \label{thm:franklahlswede}
If $\sum_{i=k+1}^n \binom{n}{i} \le m \le \sum_{i=k}^n \binom{n}{i}$, then the minimum number of disjoint pairs of sets in a family of $m$ subsets of $[n]$ is attained by some family $\cF$ with $\cup_{i > k} \binom{[n]}{i} \subset \cF \subset \cup_{i \ge k} \binom{[n]}{i}$.
\end{THM}

Note that this theorem describes the approximate structure of the extremal families, but does not specify which sets of size $k$ should be chosen.  Since every set of size $k$ is disjoint from the same number of sets in $\cup_{i > k} \binom{[n]}{i}$, in light of the above result, minimising the total number of disjoint pairs is equivalent to minimising the number of disjoint pairs amongst sets of size $k$, a problem that was first explicitly posed by Ahlswede \cite{ahlswede80} in 1980.  In the $k = 2$ case, this translates to determining which graphs have the minimum number of disjoint pairs of edges, and had in fact been resolved earlier by Ahlswede and Katona \cite{ahlkat78}.

The extremal graphs are best described by the lexicographic order.  Under this order, we say $A < B$ if $\min (A \triangle B) \in A$; that is, sets with smaller elements are preferred.  Denote by $\cL_{n,k}(m)$ the first $m$ sets in $\binom{[n]}{k}$ under the lexicographic order.  The complement of $\cL_{n,k}(m)$ is isomorphic to the corresponding initial segment of the colexicographic order, where $A < B$ if $\max (A \triangle B) \in B$.  Let $\cC_{n,k}(m)$ denote the initial segment of the colexicographic order.

\begin{THM}[Ahlswede-Katona \cite{ahlkat78}, 1978] \label{thm:ahlswedekatona}
Over all graphs on $n$ vertices with $m$ edges, either $\cL_{n,2}(m)$ or $\cC_{n,2}(m)$ minimises the number of disjoint pairs of edges.  Moreover, if $m < \frac12 \binom{n}{2} - \frac{n}{2}$, then $\cL_{n,2}(m)$ is optimal, while if $m > \frac12 \binom{n}{2} + \frac{n}{2}$, then $\cC_{n,2}(m)$ is optimal.
\end{THM}

In 2003, Bollob\'as and Leader \cite{bollead03} offered a conjecture for the solution to the problem for general $k$, which in particular implies that $\cL_{n,k}(m)$ is optimal for small families.  By taking complements, this further implies that $\cC_{n,k}(m)$ is optimal for large families.  In an earlier paper with Gan \cite{dasgansud13}, we were able to verify their conjecture for small $m$.

\subsection{Probabilistic supersaturation}

In 2012, Katona, Katona and Katona \cite{katkatkat12} introduced a probabilistic measure of supersaturation for large families.  Rather than minimising the total number of disjoint pairs in large families, they sought to maximise the probability of a random subfamily being intersecting.  More formally, given a (not necessarily uniform) family $\cF$ of sets, and some $p \in [0,1]$, let $\cF_p$ denote the random subfamily of $\cF$, where each set is retained independently with probability $p$.  For a given $0 \le m \le 2^n$, they asked for the families $\cF$ of $m$ subsets of $[n]$ maximising $\prob(\cF_p \textrm{ is intersecting})$.

Clearly, if $\cF$ is intersecting, then $\cF_p$ must also be intersecting, and hence one should take an intersecting family if possible.  Thus, as in the case of the counting supersaturation problem, one is interested in families larger than the extremal bound.

\medskip

We observe here that the probabilistic problem is in fact stronger than the counting version described before.  Indeed, by conditioning on the number of sets in $\cF_p$, we have
\begin{equation} \label{eqn:probtocount}
  \prob( \cF_p \textrm{ is intersecting}) = \sum_{t=0}^m \prob\left( \cF_p \textrm{ is intersecting} \mid \card{\cF_p} = t \right) \prob \left( \card{\cF_p} = t \right) = \sum_{t = 0}^m \inter(\cF,t) p^t (1-p)^{m-t},
\end{equation}
where $\inter(\cF,t)$ denotes the number of intersecting subfamilies of $\cF$ of size $t$.  In particular, it follows that $\inter(\cF,t) < 2^m$ for all $t$.  If we take $p = o(2^{-m})$, then $2^mp^3 = o(p^2)$, and so expanding the first few terms of the sum on the right-hand side gives
\[ \prob( \cF_p \textrm{ is intersecting}) = (1-p)^m + mp(1-p)^{m-1} + \inter(\cF,2)p^2(1-p)^{m-2} + o(p^2). \]
This quantity is maximised if and only if the number of intersecting pairs of sets in $\cF$ is maximised, and thus the number of disjoint pairs must be minimised.  Hence a solution to the probabilistic problem for all values of $p$ provides a solution to the counting problem as well.

\medskip

Katona, Katona and Katona \cite{katkatkat12} determined the extremal families for $m \le 2^{n-1} + \binom{n-1}{\left \lceil (n-3)/2 \right \rceil}$.  In particular, they showed that for all $0 \le p \le 1$, it is optimal to take all sets of size larger than $\frac{n}{2}$, with the remaining sets of size $\left \lfloor \frac{n}{2} \right \rfloor$ chosen to minimise the number of disjoint pairs.  They further conjectured the existence of a nested sequence $\cF_0 \subset \cF_1 \subset \hdots \subset \cF_{2^n}$ of families such that $\cF_m$ is the most probably intersecting family of size $m$.

In the same year, Russell \cite{russell12} provided some evidence towards this conjecture, by proving a result similar to Theorem \ref{thm:franklahlswede}, showing that there is a most probably intersecting family consisting of sets that are as large as possible.  However, in a later paper with Walters \cite{russwalt13}, they used the non-nestedness of the extremal graphs in Theorem \ref{thm:ahlswedekatona} to show that the most probably intersecting families are not nested for $\sum_{i=3}^n \binom{n}{i} \le m \le \sum_{i=2}^n \binom{n}{i}$.

\medskip

While the above results hold for non-uniform families, much less was known in the uniform setting.  By the Erd\H{o}s-Ko-Rado theorem \cite{ekr61}, when $n \ge 2k$, the largest $k$-uniform intersecting family has size $\binom{n-1}{k-1}$, a bound attained when we take all sets containing some fixed element.  We call such a structure a \emph{star}; note that for $m \le \binom{n-1}{k-1}$, $\cL_{n,k}(m)$ is a star consisting of $m$ sets containing $1$.

Hence it follows that for $m \le \binom{n-1}{k-1}$, $\cL_{n,k}(m)$ is an intersecting family, and thus a most probably intersecting family.  Once we have $m > \binom{n-1}{k-1}$, we can no longer take an intersecting family.  Katona, Katona and Katona showed in \cite{katkatkat12} that for $m = \binom{n-1}{k-1} + 1$, it is optimal to add any set to a full star, and thus $\cL_{n,k}(m)$ is again optimal.  By applying $i,j$-compressions, Russell and Walters \cite{russwalt13} were able to show that for any $m$, there is a left-compressed most probably intersecting family, but were unable to show which compressed family is optimal. 

\subsection{Our results}

We apply the shifting arguments developed in \cite{dasgansud13} to this probabilistic supersaturation for $k$-uniform set families.  In the case $k = 2$, we show that the lexicographic order provides the most probably intersecting graphs for all sizes up to $c \binom{n}{2}$, with $c$ approximately $\frac{1}{17}$.

\begin{THM} \label{thm:graphs}
For $n, \ell$ and $m$ satisfying $n \ge 32 \ell$ and $0 \le m \le \binom{n}{2} - \binom{n-\ell}{2}$, the lexicographic graph $\cL_{n,2}(m)$ is the most probably intersecting graph on $[n]$ with $m$ edges.
\end{THM}

When $k \ge 3$, the situation is rather more intricate.  The following theorem gives a rough structural description of hypergraphs maximising the number of $t$-intersecting subhypergraphs.  We say a star in $\cF$ is $\emph{full}$ if it contains $\binom{n-1}{k-1}$ sets, and \emph{almost full} if it has $(1 - o(1)) \binom{n-1}{k-1}$ sets.

\begin{THM} \label{thm:structure}
Let $k, \ell$ and $t \ge 2$ be integers, and suppose $n \ge n_0(k, \ell)$ and $\binom{n}{k} - \binom{n-\ell}{k} \le m \le \binom{n}{k} - \binom{n - \ell - 1}{k}$.  If $\cF$ is a $k$-uniform set family on $[n]$ of size $m$ maximising the number of intersecting subfamilies of size $t$, then either
\begin{itemize}
	\item[(i)] $\cF$ contains $\ell$ full stars, or
	\item[(ii)] $\cF$ consists of $\ell + 1$ almost-full stars.
\end{itemize}
\end{THM}

Using Theorem~\ref{thm:structure}, we are able to determine exactly the most probably intersecting hypergraphs of some particular sizes, as given below.

\begin{COR} \label{cor:hypergraphs}
Let $k$ and $\ell$ be integers, and suppose $n \ge n_0(k,\ell)$ and $\binom{n}{k} - \binom{n-\ell}{k} \le m \le \binom{n}{k} - \binom{n-\ell}{k} + n - \ell - k + 1$.  For this range of parameters, $\cL_{n,k}(m)$ is the most probably intersecting $k$-uniform hypergraph on $[n]$ with $m$ sets.
\end{COR}

We note that for both graphs and hypergraphs we actually prove something stronger, showing $\cL_{n,k}(m)$ simultaneously maximises the number of intersecting subfamilies of size $t$ for all $t$, as stated in Propositions \ref{prop:graphcount} and \ref{prop:exact}.  In particular, this implies the most probably intersecting families in these ranges do not depend on the underlying probability $p$.  Our proofs also extend to show the most probably intersecting hypergraphs are essentially unique.

\subsection{Outline and notation}

The remainder of this paper is organised as follows.  In Section \ref{sec:graphs}, we study the most probably intersecting graphs, proving Theorem \ref{thm:graphs}.  In Section \ref{sec:hypergraphs}, we extend these methods to hypergraphs, and prove Theorem \ref{thm:structure}.  Finally, in Section \ref{sec:conclusion}, we provide some concluding remarks and open questions.

\medskip

Our notation is fairly standard.  We denote by $[n]$ the first $n$ positive integers, and for any set $X$, we write $\binom{X}{k}$ for the subsets of $X$ of size $k$.  $\cL_{n,k}(m)$ represents the first $m$ sets in $\binom{[n]}{k}$ in the lexicographic order, while $\cC_{n,k}(m)$ is the corresponding initial segment of the colexicographic order; see the paragraph preceding Theorem \ref{thm:ahlswedekatona} for a description of these orders.

If $\cF$ is a $k$-uniform family of subsets of $[n]$, then for any vertex $i \in [n]$, we write $d_i$ for its degree; that is, the number of sets containing $i$.  A subset $X \subset [n]$ of elements \emph{covers} $\cF$ if for every set $F \in \cF$, we have $F \cap X \neq \emptyset$.  We let $\inter(\cF,t)$ denote the number of intersecting subfamilies of $\cF$ of size $t$.  We say that an intersecting family $\cG$ is \emph{trivially intersecting} if $\cap_{G \in \cG} G \neq \emptyset$, and we call such a family a $\emph{star}$ with centre $\cap_{G \in \cG} G$.

\section{Intersecting graphs} \label{sec:graphs}

In this section we prove Theorem \ref{thm:graphs}, thus showing the initial segment of the lexicographic order is the most probably intersecting graph when the graphs in question are not too dense.  We recall the statement below.

\setcounter{section}{1}
\setcounter{THM}{2}

\begin{THM}
For $n, \ell$ and $m$ satisfying $n \ge 32 \ell$ and $0 \le m \le \binom{n}{2} - \binom{n-\ell}{2}$, the lexicographic graph $\cL_{n,2}(m)$ is the most probably intersecting graph on $[n]$ with $m$ edges.
\end{THM}

\setcounter{section}{2}
\setcounter{THM}{0}

In order to prove this theorem, we use \eqref{eqn:probtocount} to convert the problem into one of counting intersecting subgraphs of a given size.  At the heart of the proof, therefore, is the following proposition, which shows that in this range of densities, $\cL_{n,2}(m)$ maximises the number of intersecting subgraphs of size $t$ for any $t$.  Proposition \ref{prop:graphcount} can be viewed as an extension of Theorem \ref{thm:ahlswedekatona} to larger intersecting subgraphs.

\begin{PROP} \label{prop:graphcount}
Suppose $t \ge 0$, and $n$ and $\ell$ satisfy $n \ge 2^{2 + 6 / (t-1)} \ell$.   Then, for any $0 \le m \le \binom{n}{2} - \binom{n-\ell}{2}$, the lexicographic graph $\cL_{n,2}(m)$ maximises $\inter(G,t)$ over all graphs $G$ on $[n]$ with $m$ edges.
\end{PROP}

Note that in the case of graphs, there are only two possible intersecting structures: the star and, when $t=3$, the triangle.  We will show that there are relatively few triangles, and hence the number of intersecting subgraphs is essentially determined by the number of stars.  By considering the central vertex of a star, we find that, for $t \ge 2$, the number of stars in a graph $G$ is given by $\sum_{i \in V(G)} \binom{d_i}{t}$.  As it is cleaner to first count only the stars, we separate this (main) case into the following proposition.

\begin{PROP} \label{prop:trivint}
Suppose $t \ge 0$, and $n$ and $\ell$ satisfy $n \ge 2^{2 + 6 / (t-1)} \ell$.  Then, for any $0 \le m \le \binom{n}{2} - \binom{n-\ell}{2}$, the lexicographic graph $\cL_{n,2}(m)$ maximises $f(G,t) = \sum_i \binom{d_i}{t}$ over all graphs $G$ on $[n]$ with $m$ edges.
\end{PROP}

We now begin by showing how Theorem \ref{thm:graphs} follows easily from Proposition \ref{prop:graphcount}.

\begin{proof}[Proof of Theorem \ref{thm:graphs}]
We wish to find a graph $G$ on $[n]$ with $m$ edges that maximises $\prob(G_p \textrm{ is intersecting})$.  Recall Equation \eqref{eqn:probtocount}:
\[ \prob(G_p \textrm{ is intersecting}) = \sum_{t=0}^m \inter(G,t) p^t(1-p)^{m-t}. \]

By Theorem \ref{thm:ahlswedekatona} for $t=2$ and Proposition \ref{prop:graphcount} otherwise, among all graphs on $[n]$ with $m$ edges, $\inter(G,t)$ is maximised by $\cL_{n,2}(m)$ for all $t \ge 0$.  Thus we have
\begin{align*}
  \prob(G_p \textrm{ is intersecting}) &= \sum_{t=0}^m \inter(G,t) p^t(1-p)^{m-t} \\
  &\le \sum_{t=0}^m \inter(\cL_{n,2}(m),t) p^t(1-p)^{m-t} \\
  &= \prob(\cL_{n,2}(m)_p \textrm{ is intersecting}),
\end{align*}
and so $\cL_{n,2}(m)$ is the most probably intersecting graph, as claimed.
\end{proof}

In the remainder of this section, we seek to prove Proposition \ref{prop:graphcount}.  We begin by dealing with the cleaner case of counting stars, namely Proposition \ref{prop:trivint}.

\begin{proof}[Proof of Proposition \ref{prop:trivint}]

Our proof is by induction on $m + t$.  Note that when $t = 0$, the statement is obvious, and for $t=1$, $f(G,1) = \sum_i \binom{d_i}{1} = \sum_i d_i = 2m$, and is thus maximised by $\cL_{n,2}(m)$ and, indeed, by any other graph with $m$ edges.

For the case $t=2$, note that since $\ell \le 2^{-2 - 6/(t-1)}n \le \frac14 n$, we have at most $\binom{n}{2} - \binom{\frac{3n}{4}}{2} < \frac12 \binom{n}{2} - \frac12 n$ edges.  Hence, by Theorem \ref{thm:ahlswedekatona}, it is known that $\cL_{n,2}(m)$ maximises the number of intersecting pairs of edges, which is precisely the quantity $f(G,2)$.

Moreover, when $m \le n-1$, it is easy to see that $\cL_{n,2}(m)$ is again optimal.  Indeed, $f(G,t) = \sum_i \binom{d_i}{t}$ counts the number of $t$-edge stars in $G$.  For $m \le n-1$, $\cL_{n,2}(m)$ is itself a star, and thus all subgraphs of $t$ edges are stars.  Clearly, $f(\cL_{n,2}(m),t) = \binom{m}{t}$ is optimal.

\medskip

Hence we may assume $t \ge 3$ and $m \ge n$.  Suppose first that $G$ is an extremal graph containing a full star; without loss of generality, we may assume it has all edges containing the vertex $n$.  Let $\tilde{G}$ be the induced subgraph of $G$ on the vertices $[n-1]$.  Note that for all $1 \le i \le n-1$, the degrees in $\tilde{G}$ are given by $\tilde{d}_i = d_i - 1$, as we lose the edge to $n$.  Thus we have
\begin{align*}
	f(G,t) &= \sum_i \binom{d_i}{t} = \sum_{i = 1}^{n-1} \binom{\tilde{d}_i+1}{t} + \binom{n-1}{t}\\
	&= \sum_{i=1}^{n-1} \left( \binom{\tilde{d}_i}{t} + \binom{\tilde{d}_i}{t-1} \right) + \binom{n-1}{t} = f(\tilde{G},t) + f(\tilde{G},t-1) + \binom{n-1}{t}.
\end{align*}

By the induction hypothesis, both $f(\tilde{G},t)$ and $f(\tilde{G},t-1)$ are maximised by $\tilde{G} = \cL_{n-1,2}(m - (n-1))$.  Adding to this the full star with centre $n$, we obtain $\cL_{n,2}(m)$, thus proving its optimality.

\medskip

Now suppose $G$ is an extremal graph with the largest possible maximum degree $\Delta$, and that $\Delta \le n-2$.  This means for any edge $e$ and vertex $i$, we can replace $e$ by an edge containing $i$.  This shifting operation, coupled with the assumption of optimality, will allow us to determine the structure of $G$, and eventually derive a contradiction.

To begin with, we establish a lower bound for $f(G,t)$.  Let $1 \le r \le \ell - 1$ be such that $\binom{n}{2} - \binom{n-r}{2} < m \le \binom{n}{2} - \binom{n-r-1}{2}$.  In this range, $\cL_{n,2}(m)$ consists of $r$ full stars and a partial star.  Thus if $G$ is extremal, we must have $f(G,t) \ge f(\cL_{n,2}(m),t) \ge  r \binom{n-1}{t} + (n-r) \binom{r}{t} > r \binom{n-1}{t}$.

We shall now double-count to deduce the existence of a high-degree vertex.  Since every star we count in $f(G,t)$ contains $t$ edges, and the number of stars an edge is in is determined by the degrees of its endpoints, we have
\[ t f(G,t) = \sum_{e = \{i,j\} \in E(G)} \left( \binom{d_i - 1}{t-1} + \binom{d_j - 1}{t-1} \right) \le 2m \binom{ \Delta - 1}{t-1}, \]
where $\Delta$ is the maximum degree in $G$.  Applying the previous lower bound on $f(G,t)$ gives $\binom{\Delta - 1}{t-1} \ge \frac{rt}{2m} \binom{n-1}{t} = \frac{r(n-1)}{2m} \binom{n-2}{t-1} > \frac14 \binom{n-2}{t-1}$, since $m < (r+1)(n-1)$.  Since $\binom{\alpha p}{q} \le \alpha^q \binom{p}{q}$ for $0 \le \alpha \le 1$ (see Lemma \ref{lem:binomial}), it follows that $\Delta > 4^{- 1 / (t-1)} (n-2) + 1 \ge 4^{- 1 / (t-1)}n - 1$.  Without loss of generality, suppose $1$ is a vertex of maximum degree.

\medskip

Consider any edge $e = \{i,j\} \in E(G)$, and suppose without loss of generality $d_i \ge d_j$.  $e$ is in $\binom{d_i - 1}{t-1} + \binom{d_j - 1}{t-1} \le 2 \binom{d_i - 1}{t-1}$ stars of size $t$.  On the other hand, if we replace $e$ with an edge containing $1$, we would create at least $\binom{\Delta}{t-1}$ new stars.  Hence, by the extremality of $G$, we must have $2 \binom{d_i -1 }{t-1} \ge \binom{\Delta}{t-1}$, so $d_i - 1 \ge 2^{-1/(t-1)} \Delta \ge 2^{ - 3 / (t-1)}n - 1$.

This implies that $X = \{x \in [n]: d_x \ge 2^{-3/(t-1)} n \}$ forms a vertex cover of $G$.  This cover cannot be too large, as we have the bound
\[ 2(r+1)(n-1) > 2m = \sum_i d_i \ge \sum_{x \in X} d_x \ge 2^{-3/(t-1)} n \card{X}, \]

and so $s = \card{X} < 2^{1 + 3/(t-1)} (r+1)$.

Moreover, let $j$ be any vertex not adjacent to $1$.  We claim that $j$ must in fact be isolated.  Suppose to the contrary there were some vertex $i \neq 1$ with the edge $\{i,j\} \in E(G)$.  This edge is contained in $\binom{d_i-1}{t-1} + \binom{d_j -1}{t-1}$ stars of $t$ edges.  If we were to replace $\{i,j\}$ with the edge $\{1,j\}$, we would create $\binom{\Delta}{t-1} + \binom{d_j-1}{t-1} > \binom{d_i-1}{t-1} + \binom{d_j-1}{t-1}$ stars, contradicting the optimality of $G$.

Thus it follows that all the edges of $G$ are supported on the $\Delta + 1$ vertices in the closed neighbourhood of $1$, and that $G$ has a cover $X$ of size $s < 2^{1 + 3/(t-1)} (r+1)$ vertices, all of which have degree at least $2^{-3/(t-1)} n$.  Note that the vertices outside the cover have degree at most $s$, as they can only be adjacent to vertices in $X$.

\medskip

To complete the argument, we shall show by shifting some edges that a graph with isolated vertices cannot be optimal.  Without loss of generality, let $X = [s]$ be the cover mentioned above, and further assume that $v$ has the lowest degree in $X$.  Note that $v$ has at least $2^{-3/(t-1)}n - (s-1) \ge s-1$ neighbours outside $X$, since $s < 2^{1 + 3/(t-1)} (r+1)$.  Let $G'$ be the graph obtained from $G$ by removing $s-1$ edges from $v$ to neighbours $N \subset X^c$, and replacing them with $s-1$ edges from a previously isolated vertex $w$ to the other $s-1$ vertices in $X$.  Note that these vertices all have degree at least $d_v$.

Comparing degrees in $G'$ to those in $G$, we find that the $s-1$ vertices in $X \setminus \{v\}$ have degree one larger, the degree of $v$ has decreased by $s-1$, the degrees of the $s-1$ vertices in $N$, which were previously at most $s$, have decreased by $1$, and $w$ now has degree $s-1$.  The change in the number of intersecting subgraphs is thus
\begin{align}
	&f(G',t) - f(G,t) \notag \\
	&= \sum_i \binom{d_i'}{t} - \sum_i \binom{d_i}{t} \notag \\
	&= \sum_{i \in X \setminus \{v\}} \left( \binom{d_i+1}{t} - \binom{d_i}{t} \right) + \left( \binom{d_v - s + 1}{t} - \binom{d_v}{t} \right) + \sum_{i \in N} \left( \binom{d_i -1 }{t} - \binom{d_i}{t} \right) + \binom{d_w'}{t} \notag\\
	&= \sum_{i \in X \setminus \{v\}} \binom{d_i}{t-1} - \sum_{j=1}^{s-1} \left( \binom{d_v-j+1}{t} - \binom{d_v - j}{t} \right) - \sum_{i \in N} \binom{d_i-1}{t-1} + \binom{s-1}{t} \notag \\
	&\ge \sum_{i \in X \setminus \{v\}} \binom{d_v}{t-1} - \sum_{j=1}^{s-1} \binom{d_v-j}{t-1} - (s-1) \binom{s-1}{t-1} \notag \\
	&= \sum_{j = 1}^{s-1} \left( \binom{d_v}{t-1} - \binom{d_v -  j}{t-1}\right) - (s-1) \binom{s-1}{t-1} \ge \sum_{j=1}^{s-1} j \binom{d_v - j}{t-2} - (s-1) \binom{s-1}{t-1} \label{ineq:graphslack} \\
	&\ge \sum_{j=1}^{s-1} j \binom{d_v - s+1}{t-2} - (s-1) \binom{s-1}{t-1} = \binom{s}{2} \binom{d_v - s+1}{t-2} - (s-1) \binom{s-1}{t-1} \notag \\
	&\ge \binom{s}{2} \binom{s+1}{t-2} - (s-1) \binom{s-1}{t-1} \qquad \textrm{[since $d_v \ge 2^{-3/(t-1)}n \ge 2^{2 + 3/(t-1)}(r+1) > 2s$]} \notag \\
	&= \binom{s}{2} \binom{s+1}{t-2} - \frac{(s-1)^2}{t-1} \binom{s-2}{t-2} \ge \binom{s}{2} \left( \binom{s+1}{t-2} - \binom{s-2}{t-2} \right) \ge 0, \notag
\end{align}
since $t \ge 3$.  Hence, by shifting edges, we can increase the maximum degree of $G$ without decreasing the objective function.  This contradicts the assumption that $G$ was optimal with the largest maximum degree.
\end{proof}

Finally, we show how to deduce the general case of Proposition \ref{prop:graphcount} from this result.  This requires only minor modifications of the above proof, which we highlight below.

\begin{proof}[Proof of Proposition \ref{prop:graphcount}]
Note that $f(G,t)$ counts precisely the number of stars of $t$ edges in the graph $G$ (except when $t=0$, when the empty graph is counted $n$ times, and $t=1$, in which case the single edges are counted twice).  When $t \neq 3$, these stars are the only intersecting graphs of $t$ edges, and thus Proposition \ref{prop:graphcount} follows directly from Proposition \ref{prop:trivint}.

\medskip

When $t=3$, we must augment the proof of Proposition \ref{prop:trivint} to also account for the triangles in the graph.  However, the number of possible triangles is a lower order term that can be taken care of by slightly altering the argument.

We begin by observing that the base case of the inductive argument still holds.  The proposition holds for $m \le n-1$, as every $3$-edge subgraph of $\cL_{n,2}(m)$ is intersecting, which is clearly the best possible.  Moreover, suppose $G$ contains a full star, and let $G'$ denote the subgraph with the full star removed.  Then each edge in $G'$ induces one triangle with edges from the full star.  Thus we can again write the number of intersecting subgraphs of $3$ edges as a constant term, independent of the structure of $G'$, plus the corresponding terms from $G'$, and can then apply the inductive hypothesis.

We next need a lower bound on the maximum degree $\Delta$.  Note that an edge $\{i,j\}$ can be involved in at most $\min \{d_i-1,d_j-1\}$ triangles, and thus in at most $\binom{d_i-1}{2} + \binom{d_j-1}{2} + \min \{d_i-1,d_j-1\}$ intersecting subgraphs of three edges in total.  Hence we have
\begin{align*}
	3 \inter(G,3) &\le \sum_{\{i,j\} \in E(G)} \left( \binom{d_i-1}{2} + \binom{d_j-1}{2} + \min \{ d_i-1,d_j-1 \} \right) \\
	&\le 2m \left( \binom{\Delta-1}{2} + \Delta - 1 \right) = 2m \binom{\Delta}{2}.
\end{align*}
On the other hand, we have $\inter(G,3) \ge \inter(\cL_{n,2}(m)) \ge r \binom{n-1}{3}$.  From these inequalities, we can deduce $\Delta (\Delta - 1) \ge \frac14(n - 2)(n-3)$.  Again, assume that $1$ is a vertex of maximum degree.

Now if we have the edge $e = \{i,j\}$ with $d_i \ge d_j$, then $e$ is contained in at most $\binom{d_i-1}{2} + \binom{d_j-1}{2} + d_j-1 \le \binom{d_i-1}{2} + \binom{d_j}{2} \le \binom{d_i-1}{2} + \binom{d_i}{2}$ intersecting families of three edges.  Since replacing $e$ with an edge containing $1$ would create at least $\binom{\Delta}{2}$ new stars of three edges, we must have $\binom{d_i-1}{2} + \binom{d_i}{2} > \binom{\Delta}{2}$, which, given our above bound on $\Delta$, shows $X = \{i : d_i \ge \frac{1}{2\sqrt{2}} n \}$ is a cover for $G$.  In fact, these shifting arguments also show that $X$ must be a clique.  As before, we can also show that if $v$ is not adjacent to $1$, then $v$ must in fact be an isolated vertex.

To complete the argument, we show that graphs with isolated vertices cannot be optimal by shifting $s-1$ edges to an isolated vertex, where $|X| = s \ge 2$.  In the proof of Proposition \ref{prop:trivint}, we saw that such a shift results in a gain of at least $\binom{s}{2} \left( \binom{s+1}{t-2} - \binom{s-2}{t-2} \right) = 3 \binom{s}{2}$ stars of three edges.  On the other hand, as $V(G) \setminus X$ is an independent set, we lose at most $(s-1)^2$ triangles, since every edge removed can only form a triangle with another vertex from $X$.  However, by adding $s-1$ edges from a clique to a new vertex, we create $\binom{s-1}{2}$ new triangles.  Hence we incur a net loss of at most $\binom{s}{2}$ triangles.  For $s \ge 3$ we have $3 \binom{s-1}{2} \ge \binom{s}{2}$, and so shifting the edges increases the maximum degree without decreasing $\inter(G,3)$, contradicting our choice of $G$.  If $s = 2$, then we are shifting one edge from the vertex of second-highest degree, say $2$, to the vertex of maximum degree.  By performing the preceding calculations more carefully, we find that we gain at least $d_2 - 2 > 0$ intersecting subgraphs of three edges, again contradicting the optimality of $G$.
\end{proof}

This completes the proof of Theorem \ref{thm:graphs}, showing that the initial segment of the lexicographic order is the most probably intersecting graph up to moderate densities.  Note that, as in all previously obtained results in \cite{katkatkat12} and \cite{russwalt13}, these graphs actually simultaneously maximise the number of intersecting subgraphs of all sizes, and hence the most probably intersecting graphs do not depend on $p$.  This phenomenon fails to hold for denser graphs, but we defer this discussion until Section \ref{sec:conclusion}.

\medskip

We conclude with some remarks on the uniqueness of the extremal graphs.  To have equality, we must in particular have equality in \eqref{ineq:graphslack}, namely that $\binom{d_v}{t-1} - \binom{d_v - j}{t-1} = j \binom{d_v-j}{t-1}$ for all $1 \le j \le s-1$.  There are only three possible cases: $t \le 2$, $t \ge d_v + 2$ or $s = 2$.  In the first case, if $t = 0$ or $t = 1$ it is trivial that there is no uniqueness, as any graph with $m$ edges will be extremal.  When $t = 2$, this reduces to the question of uniqueness in Theorem \ref{thm:ahlswedekatona}.  In this case, the extremal graphs are completely characterised by \'Abrego et al \cite{abrego09}, where it is shown that they are closely related to $\cL_{n,2}(m)$.

If $s = 2$ and $t \le d_v$, it is easy to see that shifting an edge to the vertex of highest degree increases the number of intersecting subgraphs.  For $t \ge d_v$ (and $t \le \binom{n-1}{k-1}$), the edges meeting the cover $X$ only at $v$ are not contained in any intersecting subgraphs of size $t$, and hence we may remove them to complete a star and increase the number of intersecting subgraphs.  Thus in these cases it follows that the extremal graph must contain $r$ full stars, and $\cL_{n,2}(m)$ is uniquely extremal if the number of additional edges is at least $t - r$.

\section{Intersecting hypergraphs} \label{sec:hypergraphs}

We now seek to extend these results to the hypergraph setting and prove Theorem~\ref{thm:structure}, restated below.

\setcounter{section}{1}
\setcounter{THM}{3}

\begin{THM}
Let $k, \ell$ and $t \ge 2$ be integers, and suppose $n \ge n_0(k, \ell)$ and $\binom{n}{k} - \binom{n-\ell}{k} \le m \le \binom{n}{k} - \binom{n - \ell - 1}{k}$.  If $\cF$ is a $k$-uniform set family on $[n]$ of size $m$ maximising the number of intersecting subfamilies of size $t$, then either
\begin{itemize}
	\item[(i)] $\cF$ contains $\ell$ full stars, or
	\item[(ii)] $\cF$ consists of $\ell + 1$ almost-full stars.
\end{itemize}
\end{THM}

\setcounter{section}{3}
\setcounter{THM}{0}

In contrast to the graph case, there is a rich variety of non-isomorphic intersecting structures we shall have to account for.  We call intersecting families that are not stars \emph{non-trivially intersecting}.  Despite the wide range of non-trivially intersecting families, these are very small families when $k = o(\sqrt{n})$, as the Hilton-Milner theorem \cite{hilmil67} shows that the largest non-trivially intersecting family has size $\binom{n-1}{k-1} - \binom{n-k-1}{k-1} + 1 = o\left( \binom{n-1}{k-1} \right)$.  It remains the case that most intersecting subfamilies are stars, as we show in the following lemma.

\begin{LEMMA} \label{lem:nontrivial}
For $F \in \cF$, the number of non-trivially intersecting families of size $t$ in $\cF$ containing $F$ is $O \left( n^{-t/4k} \binom{ \binom{n-1}{k-1}}{t-1} \right)$, and the total number of such families in $\cF$ is $O \left( n^{-t/4k} \binom{\binom{n-1}{k-1}}{t} \right)$.
\end{LEMMA}

While the bounds required on $n$ can be explicitly calculated, we have chosen to simplify the presentation through the use of asymptotic notation, where we fix $k$ and $\ell$ and let $n$ tend to infinity.  Note, however, that we make no assumption on the relative magnitudes of $n$ and $t$; $t$ may be as large as $\binom{n-1}{k-1}$.

The proof of Lemma \ref{lem:nontrivial} is slightly technical, and so we defer it until the end of this section.  However, throughout this section we shall require some estimates on binomial coefficients, which we collect below.

\begin{LEMMA} \label{lem:binomial}
Suppose we have integers $0 \le a \le b \le c$ and $0 < M \le S$.  Then
\begin{itemize}
	\item[(i)] $\binom{b}{r} \le \left( \frac{b}{c} \right)^r \binom{c}{r}$,
	\item[(ii)] for $r \ge 1$, if $\sum_i n_i = S$ and $0 \le n_i \le M$ for all $i$, then $\sum_i \binom{n_i}{r} \le \frac{S}{M} \binom{M}{r}$, and
	\item[(iii)] for $r \ge 2$, $\left[ \binom{b-a}{r} + \binom{c+a}{r} \right] - \left[ \binom{b}{r} + \binom{c}{r} \right] \ge \left(1 - \frac{b-a}{c} \right) \frac{ar}{c - r+1} \binom{c}{r}$.
\end{itemize}
\end{LEMMA}

\begin{proof}[Proof of Lemma \ref{lem:binomial}]
\begin{itemize}
\item[(i)]  By definition, we have
\[ \binom{b}{r} = \frac{1}{r!} \prod_{j=0}^{r-1} (b - j) \le \frac{1}{r!} \prod_{j=0}^{r-1} \frac{b}{c} (c - j) = \left( \frac{b}{c} \right)^r \binom{c}{r}. \]

\item[(ii)]  Suppose we had $i$ and $j$ such that $0 < n_j \le n_i < M$.  Fixing the other variables, we have
\begin{align*}
	\binom{n_j-1}{r} + \binom{n_i + 1}{r} &= \binom{n_j}{r} - \binom{n_j-1}{r-1} + \binom{n_i}{r} + \binom{n_i}{r-1} \\
	&= \binom{n_j}{r} + \binom{n_i}{r} + \binom{n_i}{r-1} - \binom{n_j-1}{r-1} \ge \binom{n_j}{r} + \binom{n_i}{r}.
\end{align*}
This shows we may assume there is at most one $i$ for which $0 < n_i < M$.  Since $\sum_i n_i = S$, this implies we have $m = \left \lfloor \frac{S}{M} \right \rfloor$ variables $n_j = M$, with one variable equal to $S - mM$.  Hence, using (i),
\begin{align*}
	\sum_i \binom{n_i}{r} &\le m \binom{M}{r} + \binom{S - mM}{r} \le m \binom{M}{r} + \left( \frac{S - mM}{M} \right)^r \binom{M}{r} \\
	&\le \left( m + \frac{S - mM}{M} \right) \binom{M}{r} = \frac{S}{M} \binom{M}{r}.
\end{align*}

\item[(iii)] We rearrange and telescope the sums
\begin{align*}
	&\left[ \binom{b-a}{r} + \binom{c+a}{r} \right] - \left[ \binom{b}{r} + \binom{c}{r} \right] = \left[ \binom{c+a}{r} - \binom{c}{r} \right] - \left[ \binom{b}{r} - \binom{b-a}{r} \right] \\
	&= \sum_{j=1}^a \left( \left[ \binom{c+j}{r} - \binom{c+j-1}{r} \right] - \left[ \binom{b-a+j}{r} - \binom{b-a+j-1}{r} \right] \right) \\
	&= \sum_{j=1}^a \left[ \binom{c+j-1}{r-1} - \binom{b-a+j-1}{r-1} \right].
\end{align*}
Using (i), we can estimate these differences
\begin{align*}
	\binom{c+j-1}{r-1} - \binom{b-a+j-1}{r-1} &\ge \binom{c+j-1}{r-1} - \left(\frac{b-a+j-1}{c+j-1} \right)^{r-1} \binom{c+j-1}{r-1} \\
	&\ge \left( 1 - \frac{b-a+j-1}{c+j-1} \right) \binom{c+j-1}{r-1} \ge \left(1 - \frac{b-a}{c} \right) \binom{c}{r-1}.
\end{align*}
Thus we have
\[ \left[ \binom{b-a}{r} + \binom{c+a}{r} \right] - \left[ \binom{b}{r} + \binom{c}{r} \right] \ge \left(1 - \frac{b-a}{c} \right) \sum_{j=1}^a \binom{c}{r-1} =  \left(1 - \frac{b-a}{c} \right) \frac{ar}{c-r+1} \binom{c}{r}. \]
\end{itemize}
\end{proof}

Armed with these lemmas, we may now proceed to deduce our counting result.  In particular, Lemma \ref{lem:nontrivial} implies that when counting intersecting subfamilies of size $t$, the non-trivially intersecting families are a lower order term, and so we may focus on the number of stars with $t$ edges.  Applying similar shifting arguments to those in Section \ref{sec:graphs}, we shall deduce the rough structural characterisation of optimal families given in Theorem~\ref{thm:structure}.

Before we begin to prove Theorem~\ref{thm:structure}, we first analyse the initial segment of the lexicographic order to obtain a lower bound on the number of intersecting subfamilies in an optimal family.  Note that, for $m$ in the above range, $\cL_{n,k}(m)$ consists of all sets intersecting $[\ell]$, with $m - \binom{n}{k} + \binom{n-\ell}{k}$ additional sets all containing $\ell + 1$.  Hence $\cL_{n,k}(m)$ falls under case (i) of the theorem.

When counting the intersecting subfamilies of size $t$ in $\cL_{n,k}(m)$, we consider only the stars with centre $i$ for some $1 \le i \le \ell$.  There are $\ell$ choices for the centre of the star, and then for each star we must choose $t$ of the $\binom{n-1}{k-1}$ possible sets.  A star is overcounted only if all its sets contain at least two elements from $[\ell]$, giving at most $\binom{n-2}{k-2}$ sets for each choice of elements from $[\ell]$.  By the Bonferroni Inequalities and Lemma \ref{lem:binomial}, we have
\[ \inter(\cL_{n,k}(m),t) \ge \ell \binom{ \binom{n-1}{k-1}}{t} - \binom{\ell}{2} \binom{ \binom{n-2}{k-2} }{t} \ge \left[ \ell - \binom{\ell}{2} \left( \frac{k-1}{n-1} \right)^t \right] \binom{\binom{n-1}{k-1}}{t} = \left( \ell - o(1) \right) \binom{ \binom{n-1}{k-1} }{t}. \]

This gives us a lower bound on $\inter(\cF,t)$ for any optimal family $\cF$.  We now proceed with the proof of Theorem~\ref{thm:structure}.

\begin{proof}[Proof of Theorem~\ref{thm:structure}]
Suppose $\cF$ is optimal for the given parameters.  Note that we may assume $\ell \ge 1$, as (i) is trivially satisfied for $\ell = 0$.

Let $d_i$ denote the degree of vertex $i$.  Our goal is to show that either $d_i = \binom{n-1}{k-1}$ for $\ell$ vertices $i$, or $d_i = \left(1 - o(1) \right) \binom{n-1}{k-1}$ for $\ell + 1$ vertices that cover $\cF$.  Suppose $\cF$ has $p$ full stars, which we may assume have centres $1 \le i \le p$.  If $p = \ell$ we are done, so assume $p \le \ell - 1$.

Note that for $i > p$, none of the vertices have full degree, and so we may replace any set in $\cF$ with a set containing $i$.  In order to fully utilise this shifting, we will first show there is a vertex of relatively large degree.  From this, we shall deduce the existence of a small set of vertices covering all the edges.  Finally, we shall shift sets in this small cover to obtain the desired result.

To begin, note that by optimality we must have
\[ \inter(\cF,t) \ge \inter(\cL_{n,k}(m),t) \ge \left( \ell - o(1) \right) \binom{\binom{n-1}{k-1}}{t}. \]

By Lemma \ref{lem:nontrivial}, it follows that almost all of these intersecting subfamilies should be stars.  Let $d_i$ denote the degree of vertex $i$.  Then, counting over the centres of the stars, we have
\[ \inter(\cF,t) \le \sum_i \binom{d_i}{t} + o \left( \binom{\binom{n-1}{k-1}}{t} \right) = \left( p + o(1) \right) \binom{\binom{n-1}{k-1}}{t} + \sum_{i > p} \binom{d_i}{t}, \]
and so
\[ \sum_{i > p} \binom{d_i}{t} \ge \left( \ell - p - o(1) \right) \binom{\binom{n-1}{k-1}}{t}. \]

Note that by double-counting the edges, we have $\sum_{i > p} d_i \le \sum_i d_i = km \le k (\ell+1) \binom{n-1}{k-1}$.  Suppose we had $d_i \le M = c\binom{n-1}{k-1}$ for all $i > p$.  By Lemma \ref{lem:binomial}, we have
\[ \sum_{i > p} \binom{d_i}{t} \le \frac{k (\ell + 1) \binom{n-1}{k-1}}{M} \binom{M}{t} \le k (\ell+1) c^{t-1} \binom{\binom{n-1}{k-1}}{t}. \]

Comparing this to the lower bound, we must have $k(\ell+1)c^{t-1} \ge \ell - p - o(1)$, which implies $c = \Omega(1)$.  Hence we have some vertex, which we may assume to be $i=p+1$, with $d_{p+1} \ge c \binom{n-1}{k-1}$.

We shall now show that there is a small cover of vertices of large degree.  Let $X = \left\{ i : d_i \ge \frac{c}{k} \binom{n-1}{k-1} \right\}$, and suppose for contradiction we have $F \in \cF$ with $F \cap X = \emptyset$.

We have $\{ G \in \cF: G \cap F \neq \emptyset \} = \cup_{i \in F} \{ G \in \cF : i \in G \}$, and so, since $F \cap X = \emptyset$, there are at most $\sum_{i \in F} d_i < c\binom{n-1}{k-1}$ sets in $\cF$ intersecting $F$.  Since any intersecting subfamily of size $t$ containing $F$ must consist only of sets intersecting $F$, there are fewer than $\binom{c \binom{n-1}{k-1}}{t-1}$ such subfamilies.

On the other hand, replacing $F$ with a set containing $p+1$ creates at least $\binom{d_{p+1}}{t-1} \ge \binom{c\binom{n-1}{k-1}}{t-1}$ stars of size $t$ in $\cF$.  This shift would thus increase the number of intersecting subfamilies of size $t$ in $\cF$, contradicting the optimality of $\cF$.  Hence we must have $F \cap X \neq \emptyset$ for all $F \in \cF$; that is, $X$ covers $\cF$.

We now show this cover is small.  Indeed, we have
\[ k(\ell + 1)\binom{n-1}{k-1} \ge km = \sum_i d_i \ge \sum_{i \in X} d_i \ge \frac{c}{k} \binom{n-1}{k-1} \card{X}, \]
and so $\card{X} \le \frac{k^2(\ell+1)}{c} = O(1)$, as desired.

Now take a minimal subcover in $X$, which we may assume to be $[r]$.  Thus $r \le \card{X} = O(1)$.  Since $m \ge \binom{n}{k} - \binom{n-\ell}{k}$, we must have $r \ge \ell + 1$ (we cannot have $r = \ell$, as we have assumed $\cF$ only has $p < \ell$ full stars).  Note that every vertex in $[r]$ has degree at least $\frac{c}{k} \binom{n-1}{k-1}$.  Moreover, for any vertex $i \notin [r]$, all sets containing $i$ must also meet $[r]$, and so we have $d_i \le r \binom{n-2}{k-2}$.

We shall employ shifting arguments to show that all vertices in $[r]$ that are not of full degree should have approximately equal degrees.  Indeed, let $i$ and $j$ be two such vertices.  By the minimality of the cover, there must be some set $F$ with $F \cap [r] = i$.  From the preceding remarks, it follows that the number of sets intersecting $F$ is at most $\sum_{v \in F}  d_v \le d_i + r(k-1) \binom{n-2}{k-2}$.  Hence $F$ is in at most $\binom{d_i + r(k-1)\binom{n-2}{k-2}}{t-1}$ intersecting subfamilies of size $t$.

On the other hand, if we were to add a new set containing $j$, it would be in at least $\binom{d_j}{t-1}$ stars of $t$ sets containing $j$.  By optimality, it cannot be desirable to shift $F$ to a set containing $j$, and so we must have $\binom{d_i + r(k-1) \binom{n-2}{k-2}}{t-1} \ge \binom{d_j}{t-1}$, and hence $d_j \le d_i + r(k-1) \binom{n-2}{k-2} = d_i + o \left( \binom{n-1}{k-1} \right)$.  By symmetry, we have $d_j = d_i + o \left( \binom{n-1}{k-1} \right)$ for all such vertices $i,j$.

Let us now review what we have revealed of the structure of $\cF$.  There are $p$ vertices $[p]$ of degree $\binom{n-1}{k-1}$, and a further $r-p$ vertices $[r] \setminus [p]$ of almost-equal degree that cover the remaining edges.  Let $\alpha \in [0,1]$ be such that $m = \binom{n}{k} - \binom{n-\ell}{k} + \alpha \binom{n-1}{k-1} = \left( \ell + \alpha - o(1) \right) \binom{n-1}{k-1}$.  Since the first $p$ vertices cover $\left( p - o(1) \right) \binom{n-1}{k-1}$ edges, the degrees of the remaining $r-p$ vertices must be $\frac{\ell - p + \alpha + o(1)}{r-p} \binom{n-1}{k-1}$.  Let us assume they are listed in order of decreasing degrees, so $d_{p+1} \ge d_r$.

Suppose for some fixed $0 < \varepsilon < \frac{c}{k}$ we had $\frac{\ell - p + \alpha + o(1)}{r-p} < 1 - \varepsilon$.  Since $d_r \ge \frac{c}{k} \binom{n-1}{k-1}$, and there are $o \left( \binom{n-1}{k-1} \right)$ sets containing $r$ that also contain another element of $[r]$, we can find a set of $\varepsilon \binom{n-1}{k-1}$ edges that only meet $[r]$ at $r$.  We shall shift these edges to the vertex $p+1$.

By Lemma \ref{lem:nontrivial}, the number of non-trivially intersecting subfamilies of size $t$ created or destroyed is a lower-order term, while the degrees of vertices outside $[r]$ are so small that by Lemma \ref{lem:binomial} we may ignore the number of stars with centres outside $[r]$.  Hence the only intersecting subfamilies we need to consider are the stars with centres $p+1$ or $r$.

Before the shift, we had $\binom{d_{p+1}}{t} + \binom{d_r}{t}$ such stars, and after the shift, there are $\binom{d_{p+1} + \varepsilon \binom{n-1}{k-1}}{t} + \binom{d_r - \varepsilon \binom{n-1}{k-1}}{t}$ stars.  Applying Lemma \ref{lem:binomial}, we gain at least
\begin{equation} \label{ineq:shifting}
\left( 1 - \frac{d_r - \varepsilon \binom{n-1}{k-1}}{d_{p+1}} \right) \frac{\varepsilon t \binom{n-1}{k-1}}{d_{p+1} - t + 1} \binom{d_{p+1}}{t} > \varepsilon^2 \binom{d_{p+1}}{t} 
\end{equation}
stars.  This is strictly positive unless $t > d_{p+1} \ge \frac{c}{k} \binom{n-1}{k-1}$.  In this case, it follows by the Hilton-Milner theorem \cite{hilmil67} that the only intersecting families of size $t$ are stars.  Since no set meeting the cover $X$ only in $p+1$ is contained in a star of size $t$ (as $t > d_{p+1} \ge d_i$ for any vertex $i$ in such a set), we may shift sets containing $p+1$ to other vertices in the cover.  We can repeat this process until we obtain a full star, which will strictly increase the number of $t$-stars, contradicting the optimality of $\cF$.

The positivity of~\eqref{ineq:shifting} contradicts the optimality of $\cF$.  Hence we must have $\frac{\ell - p + \alpha + o(1)}{r-p} = 1 - o(1)$.  Since $r \ge \ell + 1$, this is only possible when $r = \ell + 1$ (and $\alpha = 1 - o(1)$), and so it follows that $\cF$ consists of $\ell + 1$ almost-full stars, and thus we are in case (ii).

This completes the proof of Theorem~\ref{thm:structure}.
\end{proof}

This result provides us with the approximate structure of the extremal families.  In particular, when $\alpha$ is not $1 - o(1)$, we know that any extremal family contains $\ell$ full stars, and hence is close to $\cL_{n,k}(m)$ in structure.  In order to show that $\cL_{n,k}(m)$ is in fact optimal, it remains to determine the structure of the sets outside the $\ell$ full stars.  In some special cases, we are able to do this exactly, as given by the following proposition.

\begin{PROP} \label{prop:exact}
Let $k, \ell$ and $t$ be integers, and suppose $n \ge n_0(k,\ell)$ and $\binom{n}{k} - \binom{n-\ell}{k} \le m \le \binom{n}{k} - \binom{n-\ell}{k} + n - \ell - k + 1$.  If $\cF$ is a $k$-uniform set family on $[n]$ with $m$ edges, then $\inter(\cF,t) \le \inter(\cL_{n,k}(m),t)$.
\end{PROP}

\begin{proof}[Proof of Proposition \ref{prop:exact}]
If $t = 0$ or $t = 1$, then there is nothing to prove, as $\inter(\cF,0) = 1$ and $\inter(\cF,1) = m$ for all such families $\cF$.  Hence we may assume $t \ge 2$, and thus apply Theorem~\ref{thm:structure}.  It follows that $\cF$ must contain $\ell$ full stars.  Let us write $\cF = \cF_0 \cup \cF_1$, where $\cF_0$ is the union of the $\ell$ full stars, and $\cF_1$ consists of the remaining sets.  Let $m_1 = \card{\cF_1} = m - \binom{n}{k} + \binom{n-\ell}{k}$ denote the number of additional sets $\cF$ contains.  If $m_1 = 0$ then we are done, as all edges are accounted for.  If $m_1 = 1$, then by symmetry it does not matter which set we add outside the $\ell$ stars, and so it again follows that $\cL_{n,k}(m)$ is optimal.  Hence we may assume $m_1 \ge 2$.

We will now show that $\inter(\cF,t)$ is maximised when $\card{ \cap_{F \in \cF_1} F } = k-1$; that is, when the sets in $\cF_1$ have the maximum possible intersection.  In our case, since $m_1 \le n - \ell - k + 1$, the additional sets in $\cL_{n,k}(m)$ all share the elements $\{\ell+1, \ell+2, \hdots, \ell+k-1\}$, and hence it will follow that $\cL_{n,k}(m)$ is optimal.

We count the intersecting subfamilies of $\cF$ based on their intersection with $\cF_1$.  Given some $\cH \subset \cF_1$ with $h$ sets, let $\mathrm{ext}(\cH)$ denote the number of extensions of $\cH$ by sets in $\cF_0$ to an intersecting subfamily of $\cF$ of size $t$.  In other words, it is the number of intersecting subfamilies in $\cF_0$ of size $t-h$ that intersect all sets in $\cH$.  We then have
\[ \inter(\cF,t) = \sum_{h=0}^t \sum_{\cH \in \binom{\cF_1}{h}} \mathrm{ext}(\cH). \]

When $h=0$, we simply obtain the number of intersecting subfamilies of size $t$ in $\cF_0$, which is independent of $\cF_1$.  If $h = 1$, then by symmetry it does not matter which set we choose for $\cH$.  Hence we may assume $h \ge 2$.  Suppose we have $\card{ \cap_{H \in \cH} H } = a$.  The number of sets $F \in \cF_0$ that intersect $\cH$ without containing one of the $a$ common elements is very small.  Indeed, fix any set $H \in \cH$.  Since $F \cap H \neq \emptyset$, there are $k$ options for this intersection $x$.  As we are not selecting one of the $a$ common elements of $\cH$, there must be some other set $H' \in \cH$ not containing $x$.  Hence we must again select an element of $H'$, giving a further $k$ options at the most.  Finally, since $F$ belongs to $\cF_0$, we must choose one of the $\ell$ centres of the stars.  There are then a further $k-3$ elements to choose for $F$.  Thus there are at most $\ell k^2 \binom{n-3}{k-3} < \frac{\ell k^3}{n} \binom{n-2}{k-2}$ such sets $F$.  This will be a lower order term, which we may disregard.  In particular, this implies that we should have $a \ge 1$ for $\cH$ to have a significant number of extensions.

We shall now estimate $\mathrm{ext}(\cH)$.  Calculations similar to those in the proof of Lemma \ref{lem:nontrivial} show that the number of extensions that are not themselves stars is a lower-order term, and hence we need only consider trivially intersecting extensions.  There are three cases to consider.

The centre of the star could be one of the centres of the $\ell$ full stars in $\cF_0$.  There are thus $\ell$ choices for the centre, and then the sets chosen must intersect $\cH$.  In light of our previous remarks, the number of such sets is dominated by those containing one of the $a$ common elements, giving $(a + o(1)) \binom{n-2}{k-2}$ options.  We double-count very few extensions, as then the sets from $\cF_0$ must all contain two of the $\ell$ centres of the stars, giving at most $\binom{\ell}{2} a \binom{n-3}{k-3}$ such sets.  Thus the number of extensions of this type is $(\ell - o(1)) \binom{(a + o(1)) \binom{n-2}{k-2}}{t-h}$.

The second type of trivially intersecting extensions is that where the centre is one of the $a$ common elements of $\cH$.  These sets must then contain any one of the $\ell$ centres of the stars in $\cF_0$, and thus the number of extensions is $(a - o(1)) \binom{(\ell - o(1)) \binom{n-2}{k-2}}{t-h}$.

The final type is that where the centre $x$ of the star is neither one of the $\ell$ centres from $\cF_0$ nor one of the $a$ common elements from $\cF_1$.  These sets must then contain $x$, one of the $\ell$ centres, and some elements from $\cH$, and thus there are very few such sets.

We thus conclude that $\mathrm{ext}(\cH) = (\ell - o(1)) \binom{(a + o(1)) \binom{n-2}{k-2}}{t-h} + (a - o(1)) \binom{(\ell - o(1)) \binom{n-2}{k-2}}{t-h}$.  This is increasing in $a$, and so to maximise $\mathrm{ext}(\cH)$ we must have $a = k-1$.  However, all subfamilies $\cH$ with $a = k-1$ are isomorphic, as they consist of $h$ distinct vertices attached to a common core of $k-1$ vertices.  Hence in this case $\mathrm{ext}(\cH)$ does not depend on which sets we choose, and thus $\mathrm{ext}(\cH)$ is maximised if and only if $a = k-1$.

This completes the proof of Proposition \ref{prop:exact}.

\end{proof}

Note that for a family $\cF$ to be extremal, it should maximise $\mathrm{ext}(\cH)$ for all $\cH \subset \cF_1$.  In particular, provided $t$ is not too large, this implies that $\cF$ is extremal if and only if it contains $\ell$ full stars and, for $2 \le h \le t-1$, any collection of $h$ sets in $\cF$ outside the full stars have $k-1$ vertices in common.  When $t$ is large, we will have $\mathrm{ext}(\cH) = 0$ for all $\cH$, as it will be impossible to find $t$ sets that intersect $\cH$ and meet the centres of the sets $[\ell]$.  Hence in this case $\cF_1$ may be chosen arbitrarily, and $\cF$ is extremal if and only if it contains $\ell$ full stars.

\medskip

Unfortunately, in contrast to the graph case, this gives a rather narrow range of family sizes for which we are able to determine the extremal families exactly.  However, it is necessary to have a somewhat more restricted range, as we shall show in Section \ref{sec:conclusion} that even for $\binom{n-1}{k-1} < m < 2 \binom{n-1}{k-1}$, $\cL_{n,k}(m)$ is not always optimal.

Finally, note that the exact counting result in Proposition \ref{prop:exact} implies that for these ranges of family sizes, $\cL_{n,k}(m)$ is a most probably intersecting family, thus giving Corollary~\ref{cor:hypergraphs}.  The proof is exactly the same as the derivation of Theorem \ref{thm:graphs} from Proposition \ref{prop:graphcount}, and so we do not repeat it here.

To complete this section, we now furnish a proof of Lemma \ref{lem:nontrivial}, bounding the number of non-trivially intersecting families.

\begin{proof}[Proof of Lemma \ref{lem:nontrivial}]

We begin by bounding the total number of non-trivially intersecting families of size $t$ in $\binom{[n]}{k}$.  Given such a family $\cH$, we write $\cH = \cH_0 \cup \cH_1$, where $\cH_0$ is the largest star in $\cH$.  Note that we must have $\cH_1 \neq \emptyset$, as $\cH$ is non-trivially intersecting.  Let $S = \cap_{F \in \cH_0} F$ be the centre of $\cH_0$, and let $\cM \subset \{ F \setminus S : F \in \cH_0 \}$ be the largest matching in the sets of the star after the centre is removed.  We denote the sizes of these sets as follows: $\card{\cH_0} = t_0$, $\card{S} = s$ and $\card{\cM} = b$.  

Let us first provide some bounds on these parameters.  Clearly, $s \le k$, as $S$ is a subset of each set in the star $\cH_0$.  Moreover, we claim $b \le k$ as well.  Indeed, every set $F$ in $\cH_1$ must be disjoint from $S$, as otherwise $\cH_1 \cup \{F\}$ would form a larger star.  However, it must intersect the sets $\{ S \cup M : M \in \cM \} \subset \cH_0$, and thus it must contain one element from each of the $b$ disjoint sets in $\cM$.  Since $\card{F} \le k$, we must have $b \le k$.  An easy lower bound on $t_0$ is $t_0 \ge 2$, since any pair of sets in $\cH$ forms a star.  We in fact claim $t_0 \ge \frac{t}{k}$.  Taking any set $F \in \cH$, note that all the other sets in $\cH$ must intersect $F$.  By the pigeonhole principle, there is some element of $F$ contained in at least a $\frac{1}{k}$-proportion of the other sets, giving a star of size at least $\frac{t}{k}$, as desired.

We now construct the intersecting family $\cH$.  There are $\binom{n}{s}$ choices for the centre $S$.  We then have to select $b$ sets of size $k-s$ for the matching $\cM$.  There are $\binom{n-s}{k-s}$ options for each set, giving $\binom{\binom{n-s}{k-s}}{b}$ possible matchings $\cM$.  By the maximality of $\cM$, each of the remaining sets in $\cH_0$ must meet the $(k-s)b$ elements covered by the matching $\cM$.  Hence there are at most $(k-s)b \binom{n-s-1}{k-s-1}$ choices for each set, providing $\binom{(k-s)b \binom{n-s-1}{k-s-1}}{t_0-b}$ ways to completing $\cH_0$.  As mentioned earlier, each set in $\cH_1$ must avoid $S$ and contain at least one element from each set in $\cM$.  This leaves at most $(k-s)^b \binom{n-s-b}{k-b}$ sets, from which we have to choose $t- t_0$.  Thus the number of non-trivially intersecting families with these parameters is bounded above by
\[ \binom{n}{s} \binom{\binom{n-s}{k-s}}{b} \binom{(k-s)b \binom{n-s-1}{k-s-1}}{t_0 - b} \binom{ (k-s)^b \binom{n-s-b}{k-b}}{t-t_0}. \]

Applying the estimates in part (i) of Lemma \ref{lem:binomial}, this can be further bounded by
\begin{align*}
  &n^s \left[ \left( \frac{k}{n} \right)^{b(s-1)} \binom{ \binom{n-1}{k-1}}{b} \right] \left[ \left( \frac{k^{s+1}b}{n^s} \right)^{t_0 - b} \binom{ \binom{n-1}{k-1}}{t_0 - b} \right] \left[ \left( \frac{k^{2b-1}}{n^{b-1}} \right)^{t - t_0} \binom{ \binom{n-1}{k-1}}{t-t_0} \right] \\
    &= \frac{b^{t_0 - b} k^{2b(t-t_0-1) + (s+2)t_0 - t}}{n^{(b-1)(t-t_0-1) + s(t_0 - 1) - 1}} \binom{\binom{n-1}{k-1}}{b} \binom{\binom{n-1}{k-1}}{t_0-b} \binom{\binom{n-1}{k-1}}{t-t_0}
\end{align*}

We now simplify this expression.  Since $b,s \le k$ and $t_0 \le t$, we can easily bound the numerator above by $k^{4kt}$.  For the denominator, note that $t_0 \le t-1$, as $\cH_1 \neq \emptyset$, $s \ge 1$ and $t_0 - 1 \ge \frac{t}{2k}$, as $1 \le \frac{t_0}{2}$ and $t_0 \ge \frac{t}{k}$, giving a lower bound of $n^{t/2k - 1}$.  Thus the number of non-trivially intersecting families with parameters $s,b$ and $t_0$ is at most $n \left( \frac{k^{4k}}{n^{1/2k}} \right)^{t} \binom{ \binom{n-1}{k-1} }{b} \binom{ \binom{n-1}{k-1} }{t_0 - b} \binom{ \binom{n-1}{k-1} }{t - t_0}$.

For the total number of non-trivially intersecting families, we now sum over all $s,b$ and $t_0$, obtaining a bound of
\begin{align*}
& \sum_{s=1}^k \sum_{b=1}^k \sum_{t_0 = t/k}^{t-1} n \left( \frac{k^{4k}}{n^{1/2k}} \right)^{t} \binom{ \binom{n-1}{k-1} }{b} \binom{ \binom{n-1}{k-1} }{t_0 - b} \binom{ \binom{n-1}{k-1} }{t - t_0} \\
&\le kn \left( \frac{k^{4k}}{n^{1/2k}} \right)^{t} \sum_{0 \le b \le t_0 \le t} \binom{ \binom{n-1}{k-1} }{b} \binom{ \binom{n-1}{k-1} }{t_0 - b} \binom{ \binom{n-1}{k-1} }{t - t_0 } \\
&\le kn \left( \frac{8k^{4k}}{n^{1/2k}} \right)^{t} \binom{ \binom{n-1}{k-1}}{t}.
\end{align*} 
To obtain the last inequality, we interpret the sum of the products of the three binomial coefficients as selecting, with repetition, from a collection of $\binom{n-1}{k-1}$ objects three sets $A, B$ and $C$ whose sizes sum to $t$.  We could instead first select $t$ elements from this collection, and then for each element decide which sets among $A$,$B$ and $C$ the elements should belong to.  As the selection was with repetition, an element could belong to several of the sets, and hence there are $2^3$ choices for each element.

By symmetry, every set in $\binom{[n]}{k}$ is in the same number of non-trivially intersecting families of size $t$.  Hence, averaging over all sets, we find that each set $F \in \cF$ can be in at most
\[ tkn \left( \frac{8k^{4k}}{n^{1/2k}} \right)^t  \binom{ \binom{n-1}{k-1}}{t} / \binom{n}{k} = k^2 \left( \frac{8k^{4k}}{n^{1/2k}} \right)^t \binom{ \binom{n-1}{k-1} - 1}{t - 1} < n^{-t/4k} \binom{ \binom{n-1}{k-1} - 1}{t-1} \]
for sufficiently large $n$.

Summing over the $m$ sets $F \in \cF$, the number of non-trivially intersecting families of size $t$ in $\cF$ is no larger than
\[ mn^{-t/4k} \binom{ \binom{n-1}{k-1}-1}{t-1} / t \le (\ell + 1) n^{-t/4k} \binom{n-1}{k-1} \binom{\binom{n-1}{k-1}-1}{t-1} / t = (\ell + 1) n^{-t/4k} \binom{ \binom{n-1}{k-1}}{t}, \]
thus giving the desired bounds.

\end{proof}

\section{Concluding remarks} \label{sec:conclusion}

In this paper, we have extended the shifting arguments of \cite{dasgansud13} to determine which uniform families of sets are most probably intersecting.  To derive the probabilistic result, we studied the counting version of the problem, finding families with the maximum number of intersecting subfamilies of any given size.

\medskip 

In particular, for graphs we showed that, provided the graphs are not too dense, the initial segment of the lexicographic order $\cL_{n,2}(m)$ maximises the number of intersecting subgraphs with $t$ edges.  This leaves open the question for denser graphs, on which we provide some remarks.

\medskip

In the case $t \ge \frac{n}{2}$, it is easy to show by shifting that $\cL_{n,2}(m)$ is optimal for any $m$.  Indeed, suppose we have a graph with vertices $x,y,z$ of degrees $d_x \le d_y \le d_z < n-1$, and suppose $\{x,y\}$ is an edge of the graph.  The number of stars this edge is contained in is $\binom{d_x-1}{t-1} + \binom{d_y - 1}{t-1}$.  On the other hand, if we were to add an edge containing $z$, it would be contained in at least $\binom{d_z}{t-1}$ stars.  Since $t \ge \frac{n}{2}$, we have $t - 1 \ge \frac{n-2}{2} > \frac{d_x - 1}{2}$, and so
\[ \binom{d_x - 1}{t-1} + \binom{d_y - 1}{t-1} \le \binom{d_x - 1}{t-2} + \binom{d_y - 1}{t-1} \le \binom{d_y - 1}{t-2} + \binom{d_y-1}{t-1} = \binom{d_y}{t-1} \le \binom{d_z}{t-1}. \]

Hence we may always shift edges to the vertex of highest degree until that star is filled.  Repeating the process for the remaining vertices, we obtain a graph isomorphic to $\cL_{n,2}(m)$, and hence $\cL_{n,2}(m)$ maximises $\inter(G,t)$ over all graphs $G$ with $m$ edges.

\medskip

By the theorem of Ahlswede-Katona \cite{ahlkat78}, we know for $m \ge \frac12 \binom{n}{2} + \frac{n}{2}$, the number of intersecting pairs of edges is maximised not by $\cL_{n,2}(m)$, but by its complement, $\cC_{n,2}(m)$.  Hence for such $m$ we cannot hope to have one graph $G$ that simultaneously maximises the number of intersecting subgraphs of all given orders.  Referring to Equation \eqref{eqn:probtocount}, it follows that in this regime the most probably intersecting graph depends on the probability $p$.  For very small values of $p$, $\cC_{n,2}(m)$ is optimal, while for very large values of $p$, $\cL_{n,2}(m)$ is better.

\medskip

However, the convexity of the binomial coefficients (see, for instance, Lemma \ref{lem:binomial}), suggests that if $\cL_{n,2}(m)$ maximises $\inter(G,t)$, then it should maximise $\inter(G,t')$ for all $t' \ge t$.  In particular, we believe that the result in Theorem \ref{thm:graphs} should extend to $m \le \frac12 \binom{n}{2} - \frac{n}{2}$.

\medskip

In the case of hypergraphs, the situation is even more intricate.  We showed that when $\binom{n}{k} - \binom{n-\ell}{k} \le m \le \binom{n}{k} - \binom{n-\ell}{k} + n - \ell - k + 1$, $\cL_{n,k}(m)$ maximises $\inter(\cF,t)$.  Thus we are able to determine the extremal families for the counting problem for a number of isolated ranges of family sizes.  One might hope that, as in the graph case, $\cL_{n,k}(m)$ remains optimal between these ranges as well.  However, we show now that this is not the case.

\medskip

Suppose, for simplicity, that we are counting the number of intersecting subfamilies of size three in $3$-uniform hypergraphs, whose number of edges is between one and two full stars.  Then $m = \binom{n-1}{2} + m'$, where $0 \le m' \le \binom{n-2}{2}$.  Provided we do not have two almost-full stars, Theorem~$\ref{thm:structure}$ shows that any extremal family is of the form $\cF = \cF_0 \cup \cF_1$, where $\cF_0$ is a full star, and $\cF_1$ consists of the remaining $m'$ sets.

There are four types of intersecting subfamilies of three sets: those with $0$, $1$, $2$ and $3$ sets from $\cF_1$ respectively.  To maximise the number of subfamilies with $3$ sets from $\cF_1$, it suffices to take $\cF_1$ to be intersecting.  The number of subfamilies with $0$ and $1$ sets from $\cF_1$ is independent of the structure of $\cF_1$.  Finally, to maximise the number of subfamilies with two sets from $\cF_1$, it follows from the calculations in Proposition \ref{prop:exact} that we should seek to maximise the number of pairs of sets in $\cF_1$ that intersect in two elements.

Note that in $\cL_{n,3}(m)$, the sets in $\cF_1$ all share a common element.  If we remove this common element, $\cF_1$ will be the lexicographic graph with $m'$ edges.  Since we have removed a common element from each set, we are trying to maximise the number of pairs of intersecting edges.  By the result in \cite{ahlkat78}, if $m' > \frac12 \binom{n-2}{2} + \frac{n-2}{2}$, this maximum is attained by the colexicographic graph instead, and hence it follows that $\cL_{n,3}(m)$ does not maximise $\inter(\cF,t)$.

\medskip

This phenomenon holds in general, and shows that determining the exact optimal $k$-uniform families for all $\binom{n}{k} - \binom{n-\ell+1}{k} \le m \le \binom{n}{k} - \binom{n - \ell}{k}$ may require a complete solution to the counting problem for the number of $t$-intersecting subfamilies of a $(k-1)$-uniform set family.  Indeed, it further suggests that even in this initial range, there may not be one set family that simultaneously maximises the number of intersecting subfamilies of any given size, and thus the optimal families may depend on the probability $p$.

\medskip

Finally, as with the results in \cite{katkatkat12}, \cite{russell12} and \cite{russwalt13}, the extremal families we obtain here are simultaneously optimal for the counting problems as well, and thus we use Equation \eqref{eqn:probtocount} to resolve the probabilistic problem.  It would be very interesting to develop techniques to attack the probabilistic problem directly, as one might then find a complete solution even in the regime where the optimal family depends on the underlying probability $p$.

\end{document}